\documentclass[a4paper,10pt]{article}
\usepackage[latin1]{inputenc}
\usepackage[T1]{fontenc}
\usepackage{amsmath}
\usepackage{amsfonts}
\usepackage{amstext}
\usepackage{amsthm}
\usepackage[all]{xy}
\usepackage{geometry}
\usepackage{srcltx}
\usepackage{graphics}
\usepackage{epsfig}
\usepackage{subfigure}
\usepackage{slashed}
\usepackage{color}
\usepackage{amssymb}
\usepackage{srcltx}
\newtheorem{theorem}{Theorem}[section]

\newtheorem{rem}[theorem]{Remark}

\DeclareMathOperator{\im}{im}

\DeclareMathOperator{\Hom}{Hom}
\DeclareMathOperator{\GL}{GL}

\DeclareMathOperator{\Ext}{Ext}

\title{A remark on singular sets of vector bundle morphisms}
\author{Maciej Starostka and Nils Waterstraat}

\begin{document}
\date{}
\maketitle

\footnotetext[2]{N. Waterstraat was supported by the Berlin Mathematical School and the SFB 647 ``Space--Time--Matter''.}

\begin{abstract}
If a characteristic class for two vector bundles over the same base space does not coincide, then the bundles are not isomorphic. We give under rather common assumptions a lower bound on the topological dimension of the set of all points in the base over which a morphism between such bundles is not bijective. Moreover, we show that this set is topologically non-trivial.
\end{abstract}

\section{Introduction}
Let $E$ and $F$ be real or complex vector bundles of the same dimension over a common base space $X$ and let $L:E\rightarrow F$ be a vector bundle morphism, that is, a fibrewise linear continuous map. For $\lambda\in X$, let us denote by $E_\lambda$, $F_\lambda$ the fibres of the bundles over $\lambda$. We consider the set

\[\Sigma(L)=\{\lambda\in X:\,L_\lambda\notin\GL(E_\lambda,F_\lambda)\}\subset X\]
of elements in $X$ over which $L$ is not an isomorphism, which we call the singular set of $L$. From the continuity of the determinant, it is clear that $\Sigma(L)$ is closed. Moreover, $\Sigma(L)=\emptyset$ if and only if $L$ is an isomorphism, and $\Sigma(L)=X$ can occur, for example, if $L$ maps $E$ into the zero section in $F$. We assume that the base space $X$ is a closed manifold which, depending on the characteristic class that we consider, may also assumed to be orientable. Our results show that the non-equality of some characteristic class for two bundles $E$ and $F$ may impose strong restrictions on the shape and size of $\Sigma(L)$, which even hold independently of the particular morphism $L$.\\
We discuss our results in two separate sections. In Section 2 we consider characteristic classes that are defined for any vector bundle over $X$ and prove our theorem which we subsequently apply to Stiefel-Whitney and Pontryagin classes for real vector bundles, and to Chern classes in the complex case. In Section 3 we consider orientation preserving morphisms on real vector bundles and characteristic classes which are defined for oriented bundles. We show that the proof of our theorem from the second section carries over to this setting and we apply it to the Euler class and to Wu classes.

\subsubsection*{Acknowledgements}
The ideas that are contained in this paper were developed while the authors enjoyed the inspiring atmosphere of the ``Stefan Banach International Mathematical Center'' in Warsaw. Moreover, we would like to thank Thomas Schick for valuable hints and suggestions on an early version of our results.


\section{Main theorem and examples}
In what follows we assume that the base space $X$ is a (topological) manifold of dimension $n$ and that $R$ is a principal ideal domain.  We denote by $H^k(X;G)$ the $k$-th singular cohomology group and by $\check{H}^{k}(X;G)$ the $k$-th \v{C}ech cohomology group of $X$ for $k\in\mathbb{Z}$ and some $R$-module $G$. Before we state our main theorem, we want to introduce some notation.\\
First, by an $H^k(\cdot;G)$-valued characteristic class we mean a map $c$ which assigns to any vector bundle $E$ over a manifold $X$ a cohomology class $c(E)\in H^k(X;G)$ such that $f^\ast c(E_2)=c(E_1)$ if $f:X\rightarrow Y$ is a continuous map which is covered by a bundle map $E_1\rightarrow E_2$ (cf. \cite[p. 69]{MiSta}). Note that we obtain the immediate properties

\begin{enumerate}
	\item[i)] $c(f^\ast E)=f^\ast c(E)\in H^k(Y;G)$ for any continuous map $f:Y\rightarrow X$,
	\item[ii)]  $c(E)=c(F)$ if $E$ and $F$ are isomorphic bundles over the same base $X$.
\end{enumerate}

Second, the covering dimension $\dim A$ of a topological space $A$ is the minimal value of $n\in\mathbb{N}$ such that every finite open cover of $A$ has a finite open refinement in which no point is included in more than $n+1$ elements. It coincides with the corresponding definitions of dimensions for finite CW-complexes and compact manifolds (cf. \cite{Fritsch}, \cite{Dimension}). Moreover, the covering dimension of a space $A$ is related to its \v{C}ech cohomology groups: if $\check{H}^{k}(A;G)$ is non-trivial for some $k\in\mathbb{N}$, then $k$ is a lower bound on $\dim A$ (cf. \cite[Prop. V.1.2]{Dimension}).\\
Third, recall that for every $\lambda\in X$ we have $H_n(X,X\setminus\{\lambda\};R)\cong R$. A local $R$-orientation $o(\lambda)$ at $\lambda$ is a generator of the $R$-module $H_n(X,X\setminus\{\lambda\};R)$ and $X$ is called $R$-orientable if there exists a continuous local $R$-orientation in the following sense: there is an open cover $\{U_i\}_{i\in I}$ of $X$ and elements $\mu_{U_i}\in H_n(X,X\setminus U_i;R)$ which are mapped by

\[\iota_\ast: H_n(X,X\setminus U_i;R)\rightarrow H_n(X,X\setminus\{\lambda\};R)\]
to $o(\lambda)$ for all $\lambda\in U_i$ and $i\in I$, where $\iota: (X,X\setminus U_i)\rightarrow (X,X\setminus\{\lambda\})$ denotes the inclusion.
$X$ is said to be orientable if it is $\mathbb{Z}$-orientable, which is the case if and only if $H_n(X;\mathbb{Z})\cong\mathbb{Z}$. If $X$ is smooth, then this notion of orientability coincides with the common definition from differential topology.\\
Finally, the universal coefficient theorem for cohomology states that for every $k\in\mathbb{Z}$ there is an exact sequence

\[0\rightarrow\Ext^1_R(H_{k-1}(X;R),G)\xrightarrow{\beta^k} H^k(X;G)\xrightarrow{\alpha^k}\Hom_R(H_k(X;R),G)\rightarrow 0.\]
In what follows, we write as usual $\alpha^k(\eta)\xi=\langle\eta,\xi\rangle$ for $\eta\in H^k(X;G)$ and $\xi\in H_k(X;R)$.\\  
Now our theorem reads as follows:

\begin{theorem}\label{theorem}
Let $E$ and $F$ be real or complex bundles over the closed $R$-orientable manifold $X$ of dimension $n$. Let $L:E\rightarrow F$ be a bundle morphism and let $c$ be a characteristic class having values in $H^k(\cdot;G)$. If 

\begin{align}\label{assumption}
c(E)-c(F)\notin\im\beta^k\subset H^k(X;G),
\end{align}
then there exists a cohomology class in $\check{H}^{n-k}(X;R)$ which restricts to a non-trivial class in\linebreak $\check{H}^{n-k}(\Sigma(L);R)$. In particular, $\dim\Sigma(L)\geq n-k$.
\end{theorem}

\begin{proof}
By the universal coefficient theorem and \eqref{assumption}, we can find $0\neq\xi\in H_{k}(X;R)$ such that $\langle c(E)-c(F),\xi\rangle\neq 0\in G$. Let $\eta\in\check{H}^{n-k}(X;R)$ denote its Poincar\'e dual, where we use that $X$ is compact and $R$-oriented.  According to \cite[Cor. VI.8.4]{Bredon}, we have a commutative diagram

\begin{align}\label{exact}
\begin{split}
\xymatrix{&\check{H}^{n-k}(X;R)\ar[r]^(.47){i^\ast}&\check{H}^{n-k}(\Sigma(L);R)\\
H_{k}(X\setminus\Sigma(L);R)\ar[r]^(.57){j_\ast}&H_{k}(X;R)\ar[u]\ar[r]^(.37){\pi_\ast}&H_{k}(X,X\setminus\Sigma(L);R)\ar[u]
}
\end{split}
\end{align}
where the vertical arrows are isomorphisms given by Poincar\'e-Lefschetz duality and the lower horizontal sequence is part of the long exact homology sequence of the pair $(X,X\setminus\Sigma(L))$. Because of the commutativity, the class $i^\ast\eta$ is dual to $\pi_\ast\xi$ and we now assume by contradiction the triviality of the latter one.\\
Then, by exactness, there exists $\zeta\in H_{k}(X\setminus\Sigma(L);R)$ such that $\xi=j_\ast\zeta$. Since the pullback $j^\ast L:j^\ast E\rightarrow j^\ast F$ is an isomorphism of the bundles over the manifold $X\setminus\Sigma(L)$, we obtain from the properties of $c$

\begin{align}\label{equality}
c(j^\ast E)=c(j^\ast F).
\end{align}
Consequently,

\begin{align*}
0=\langle c(j^\ast E)-c(j^\ast F),\zeta\rangle=\langle j^\ast(c(E)-c(F)),\zeta\rangle=\langle c(E)-c(F),j_\ast\zeta\rangle=\langle c(E)-c(F),\xi\rangle\neq 0,
\end{align*}
which is a contradiction. Hence $\pi_\ast\xi$ and so $i^\ast\eta\in\check{H}^{n-k}(\Sigma(L);R)$ is non-trivial.
\end{proof}

\begin{rem}
The idea to use Poincar\'e-Lefschetz duality and \eqref{exact} for estimating the topological dimension of a ``singular set'' from below is taken from topological bifurcation theory for nonlinear operator equations (cf. \cite[Thm. 3.1]{FiPejsachowiczII}).
\end{rem}
Let us now discuss some examples. The Stiefel-Whitney classes $w_k(E)\in H^k(X;\mathbb{Z}_2)$ (cf. \cite[\S 4]{MiSta}) are defined for real bundles $E$. Since $G=\mathbb{Z}_2$ is a field, $\Ext^1_{\mathbb{Z}_2}(H_{k-1}(X;\mathbb{Z}_2),\mathbb{Z}_2)$ vanishes and hence \eqref{assumption} holds if and only if $w_k(E)\neq w_k(F)$. Moreover, every manifold $X$ is $\mathbb{Z}_2$-orientable. Of particular interest are the first and second Stiefel-Whitney classes. A vector bundle $E$ is orientable if and only if $w_1(E)$ is trivial (cf. the subsequent section). Hence for a morphism $L$ between an orientable and a non-orientable bundle, the singular set $\Sigma(L)$ is at least of dimension $n-1$. Let us now assume that $L$ is a morphism between two orientable bundles $E$ and $F$, and let us recall that an orientable bundle admits a spin structure if and only if its second Stiefel-Whitney class vanishes. We obtain from our theorem that the dimension of $\Sigma(L)$ is at least $n-2$ if $w_2(E)=0$ and $w_2(F)\neq 0$, or vice versa.\\
Further characteristic classes for real vector bundles are the Pontryagin classes $p_k(E)\in H^{4k}(X;\mathbb{Z})$ (cf. \cite[\S 15]{MiSta}). Since $G=\mathbb{Z}$ in this case, we need to require that the base manifold $X$ is orientable in order to apply Theorem \ref{theorem}. Moreover, in contrast to the Stiefel-Whitney classes, $\Ext^1_\mathbb{Z}(H_{4k-1}(X;\mathbb{Z}),\mathbb{Z})$ is non-trivial in general, and so it is more difficult to check assumption \eqref{assumption}. Finally, note that the best result on the dimension that we can obtain is $\dim\Sigma(L)\geq n-4$ if $p_1(E)-p_1(F)\notin\im\beta^4$.\\
Characteristic classes for complex bundles $E$ and $F$ are the Chern classes $c_k(E)\in H^{2k}(X;\mathbb{Z})$ (cf. \cite[\S 14]{MiSta}). As for Pontryagin classes, $X$ need to be orientable in order to apply Theorem \ref{theorem} and $\Ext^1_\mathbb{Z}(H_{2k-1}(X;\mathbb{Z}),\mathbb{Z})$ is non-trivial in general. Here the best result on the covering dimension that we can get is $\dim\Sigma(L)\geq n-2$ if $c_1(E)-c_1(F)\notin\im\beta^2$.

\section{Oriented bundles}
In this section we consider real oriented vector bundles $E$, $F$ over a manifold $X$. Let us recall that a vector bundle $E$ is orientable if there exists a function which assigns an orientation to each fibre in such a way that near each point of the base there is a local trivialization carrying the orientations of fibres into the standard orientation of $\mathbb{R}^n$. By an $H^k(\cdot;G)$-valued characteristic class for oriented bundles we mean a map $e$ which assigns to any oriented vector bundle $E$ over a manifold $X$ a cohomology class $e(E)\in H^k(X;G)$ such that $f^\ast e(E_2)=e(E_1)$ if $f:X\rightarrow Y$ is a continuous map which is covered by an orientation preserving bundle map $E_1\rightarrow E_2$. We call a bundle morphism $L:E\rightarrow F$ orientation preserving if $\Sigma(L)\neq\emptyset$ and $j^\ast L:j^\ast E\rightarrow j^\ast F$ is orientation preserving, where $j:X\setminus\Sigma(L)\hookrightarrow X$ denotes the canonical inclusion. Note that if $L$ is an isomorphism, then $\Sigma(L)=X$ and our definition coincides with the usual one for bundle isomorphisms. Theorem \ref{theorem} for oriented bundles reads as follows:

\begin{theorem}\label{theoremII}
Let $E$ and $F$ be real oriented bundles over the closed $R$-orientable manifold $X$ of dimension $n$. Let $L:E\rightarrow F$ be an orientation preserving bundle morphism and let $e$ be a characteristic class for oriented bundles having values in $H^k(\cdot;G)$. If 

\begin{align}\label{assumptionII}
e(E)-e(F)\notin\im\beta^k\subset H^k(X;G),
\end{align}
then there exists a cohomology class in $\check{H}^{n-k}(X;R)$ which restricts to a non-trivial class in\linebreak $\check{H}^{n-k}(\Sigma(L);R)$. In particular, $\dim\Sigma(L)\geq n-k$.
\end{theorem}

\begin{proof}
The proof is verbatim as for Theorem \ref{theorem}. One only has to note that $j^\ast L:j^\ast E\rightarrow j^\ast F$ is by definition an orientation preserving bundle morphism and so the corresponding equality \eqref{equality} holds.
\end{proof}

As our notation suggests, an example of a characteristic class for oriented bundles is the Euler class $e(E)\in H^r(X;\mathbb{Z})$ (cf. \cite[\S 9]{MiSta}), where $r$ is the fibre dimension of the bundle $E$. A particular example appears if we consider endomorphisms $L:E\rightarrow E$ of oriented bundles having a non-trivial Euler class $e(E)\neq 0\in H^r(X;\mathbb{Z})$. Let us assume that $X$ is orientable and $H_{r-1}(X;\mathbb{Z})$ is free. Since $\Ext^1_\mathbb{Z}(H_{r-1}(X;\mathbb{Z}),\mathbb{Z})=0$ in this case, we have $\im\beta^r=0$ and conclude that $H^r(X;\mathbb{Z})\cong\Hom_\mathbb{Z}(H_r(X;\mathbb{Z}),\mathbb{Z})$ is torsion-free. From $e(-E)=-e(E)$ (cf. \cite[Prop. 9.3]{MiSta}), where $-E$ denotes the bundle $E$ with the reversed orientation, we obtain $\dim\Sigma(L)\geq n-r$ for any endomorphism which is orientation reversing over $X\setminus\Sigma(L)$.\\
Further examples to which Theorem \ref{theoremII} can be applied are Wu classes $q_k(E)\in H^{2(p-1)k}(X;\mathbb{Z}_p)$, where $p$ is an odd prime (cf. \cite[\S 19]{MiSta}). Since $\Ext^1_{\mathbb{Z}_p}(H_{2(p-1)k-1}(X;\mathbb{Z}_p),\mathbb{Z}_p)$ vanishes, again \eqref{assumptionII} holds if and only if $q_k(E)\neq q_k(F)$. However, depending on the prime $p$, the best estimate on the dimension that we can obtain is $\dim\Sigma(L)\geq n-2(p-1)$ if $q_1(E)\neq q_1(F)$.

\vspace{1cm}
Maciej Starostka\\
Institute of Mathematics\\
Polish Academy of Sciences\\
00-956 Warsaw\\
ul. Sniadeckich 8\\
and\\
Gdansk University of Technology\\
80-233 Gdañsk,\\
ul. Gabriela Narutowicza 11/12 \\
Poland\\
E-mail: maciejstarostka@gmail.com

\vspace{1cm}
Nils Waterstraat\\
Institut für Mathematik\\
Humboldt Universität zu Berlin\\
Unter den Linden 6\\
10099 Berlin\\
Germany\\
E-mail: waterstn@math.hu-berlin.de

\end{document}